\documentclass{article}

\usepackage{hyperref}

\usepackage{amsthm}
\usepackage{amssymb}

\usepackage{fixmath}

\usepackage[T1]{fontenc}
\usepackage{ucs}
\usepackage[utf8x]{inputenc}
\usepackage[english]{babel}

\usepackage{microtype}
\usepackage{babel}

\usepackage{authblk}

\usepackage{mathtools} 

\theoremstyle{theorem}
\newtheorem{theorem}{Theorem}
\newtheorem{proposition}[theorem]{Proposition}

\theoremstyle{definition}

\newtheorem{remark}[theorem]{Remark}

\providecommand{\Prob}[1]{\mathbb{P}\left\{#1\right\}}

\providecommand{\card}[1]{\texttt{\#}#1}

\newcommand{\Ex}{\mathbb{E}}

\begin{document}

\title{L\'evy's second arcsine law\\ via the ballot theorem}
\markright{L\'evy's second arcsine law via the ballot theorem}
\author{Helmut H.~Pitters\thanks{Mathematics Institute, University of Mannheim, Germany. E-mail:~\nolinkurl{helmut[dot]pitters[at]gmail[dot]com}}}

\maketitle

\begin{abstract}
We provide a new and elementary proof of L\'evy's second arcsine law for Brownian motion. The only tools required are basic properties of Brownian motion and Poisson processes, and the ballot theorem. Our proof is readily extended to Brownian motion with drift.
\end{abstract}

\section*{Introduction.}
Let $(B_t)_{t\geq 0}$ denote standard linear Brownian motion started in $B_0=0$ almost surely, and for fixed $t>0$ consider the proportion of time
\begin{align}\label{def:occupation_time}
  t^{-1}\int_0^t \mathbf 1_{(0, \infty)}(B_t) dt,
\end{align}
that Brownian motion spends above zero during $[0, t]$. Here $\mathbf 1_A(x)$ is the indicator of some set $A$ that equals one if $x\in A$ and zero otherwise. In his celebrated result from 1939, Paul L\'evy~\cite{Levy1939} characterised the distribution of~\eqref{def:occupation_time}. This result is sometimes referred to as the (L\'evy's) second arcsine law for Brownian motion, cf.~\cite[Theorem 5.28]{MoertersPeres2010}, and reads as follows.

\begin{theorem}[L\'evy's second arcsine law for Brownian motion]\label{thm:arcsine_law}
For any fixed $t>0$ the occupation time~\eqref{def:occupation_time} of Brownian motion has an arcsine distribution with support $(0, 1)$ and density
\begin{align}\label{eq:arcsine_density}
  x\mapsto \frac{1}{\pi\sqrt{x(1-x)}}\mathbf 1_{(0,1)}(x).
\end{align}
\end{theorem}

 In his proof L\'evy employed the inverse of the continuous local time process of Brownian motion and argued with the fact that this process is a stable subordinator with index $1/2,$ cf.~\cite{PitmanYor1992}. Since then various proofs of Theorem~\ref{thm:arcsine_law} have been offered. Another proof may be obtained by viewing Brownian motion as the diffusion limit of suitably chosen random walks as justified by Monroe Donsker's invariance principle, cf.~\cite[Proof of Theorem 5.28, p.~139]{MoertersPeres2010}. A more direct derivation is obtained from the Feynman-Kac formula, cf.~\cite[first application of Theorem 7.43]{MoertersPeres2010}. For proofs of Theorem~\ref{thm:arcsine_law} mainly using tools from excursion theory, see e.g.,~\cite[Chapter 4]{YenYor2013}.

In this note we offer a new and elementary proof of Theorem~\ref{thm:arcsine_law} that features, somewhat surprisingly, the (first) ballot theorem, sometimes referred to as Joseph Bertrand's ballot theorem. The proof only requires familiarity with basic properties of Brownian motion (cf.~\cite{MoertersPeres2010}), the Poisson process (cf.~\cite{Kingman1993}), and uniform order statistics (cf.~\cite[Section 4.1]{Pyke1965}).
\smallskip

Let us recall the ballot theorem, more specifically, the so-called first ballot theorem due to Betrand, cf.~\cite{Addario-BerryReed2008}.

\begin{theorem}[Bertrand~\cite{Bertrand1888}]\label{thm:ballot}
  Suppose that two candidates have been submitted to a vote. Candidate A obtains $a$ votes; candidate B obtains $b$ votes. Candidate A is elected, i.e.~$a>b.$ Then the probability that during the counting of the votes, the number of votes for A is at all times greater than the number of votes for B is $(a-b)/(a+b)$.
\end{theorem}
 Moreover, in the weak version of the ballot theorem ties are allowed, i.e.,~one only asks that during the counting of the votes the number of votes for A is at all times at least the number of votes for B. This probability is $(a+1-b)/(a+1)$, cf.~\cite{Renault2007}. For more information on the background of the ballot theorem and its various proofs the reader is referred to~\cite{Renault2007}, and to~\cite{Addario-BerryReed2008} for a probabilistic perspective. In our following proof of Theorem~\ref{thm:arcsine_law} we will in fact employ the weak version of the ballot theorem.

\section*{New proof of L\'evy's second arcsine law for Brownian motion.}
In order to prove the second arcsine law for Brownian motion we may assume $t=1$ without loss of generality due to the scaling property of $(B_t)$. Suppose that we were to `guess' the amount of time that $(B_t)$ spends above zero during $[0, 1]$, and to this end we were allowed to sample $(B_t)$ at $m$ instances chosen according to our liking. It seems natural to choose the times $U_1, \ldots, U_m$ independently (and independently of $(B_t)$) and uniformly distributed in $(0, 1),$ and to estimate said proportion by $\card\{1\leq k\leq m\colon B_{U_k}>0  \}/m$. In fact, it turns out that the probability of $\{ B_{U_1}>0, \ldots, B_{U_m}>0 \}$ agrees with the $m$th moment of the occupation time in~\eqref{def:occupation_time}, which is uniquely determined by its moments.

\begin{proposition}[Sampling the occupation time]\label{prop:sampling}
 Let $(U_k)_{k\geq 1}$ denote an i.i.d.~sequence of uniform $(0, 1)$ r.v.s that is independent of the Brownian motion $(B_t).$ Then, for any natural number $m$,
  \begin{align}
    \Ex \left [\left (\int_0^1 \mathbf 1_{(0, \infty)}(B_t)  dt \right )^m\right ] = \Prob{B_{U_1}>0, \ldots, B_{U_m}>0}  .
  \end{align}
\end{proposition}
\begin{proof}
      Let $f$ denote some real-valued measurable function. Then
      $$
      \Ex\big[f(B_{U_1})\cdots f(B_{U_m})\big] =  \Ex\left[ \int_0^1\cdots \int_0^1 f(B_{u_1}) \cdots f(B_{u_m}) d u_1 \cdots d u_m \right].
      $$
      Notice that the integrals on the right hand side are identical. Setting  $f(x):=\mathbf 1_{(0, \infty)}(x)$ shows the claim.
\end{proof}

Proposition~\ref{prop:sampling} is a special case of a more general result on the occupation time of some stochastic process that was given in~\cite[Proposition 1]{AurzadaDoeringPitters2024}. Since $\int_0^1 \mathbf 1_{\{ B_t>0 \}}  dt$ is bounded, its distribution is uniquely determined by its moment sequence. Thus we are left to compute the persistence probability
\begin{align}\label{eq:persistence}
  \Prob{B_{U_1}>0, ..., B_{U_m}>0}.
\end{align}
We now work out~\eqref{eq:persistence} by revealing a deeper connection between $(B_{U_{k}})_{k=1}^m$ and two Poisson processes in two steps.


First step. In this first step, we replace the $(U_k)$ in~\eqref{eq:persistence} by the arrival times $0=T'_0<T'_1<\cdots$ in a Poisson process $\Pi'$ of unit intensity on the positive half-line independent of $(B_t)$. To this end, we first need to recall the notion of order statistics together with a basic property of uniform order statistics. Let $x_{(1)}, \ldots, x_{(m)}$ denote the values in increasing order of some fixed but arbitrary real numbers $x_1, \ldots, x_m$. More precisely, we require that $\{ x_{(k)}\colon 1\leq k\leq m  \} = \{ x_k\colon 1\leq k\leq m  \},$ and that $x_{(1)}\leq \cdots\leq x_{(m)}.$ We call the $x_{(1)}, \ldots, x_{(m)}$ the order statistics of $x_1, \ldots, x_m$. By a standard result on order statistics (and their spacings), the order statistics $(U_{(1)}, \ldots, U_{(m)})$ of i.i.d.~uniform $(0, 1)$ r.v.s may also be represented (in distribution) as $(T'_1, \ldots, T'_m)/T'_{m+1},$ and the latter random vector is independent of $T'_{m+1},$ cf.~\cite[Sections 4.1--4.3]{Pyke1965}. Moreover, by the scaling property of Brownian motion the processes $(B(Tt)/\sqrt T)_{t\geq 0}$ and $(B_t)_{t\geq 0}$ are equal in distribution for any fixed $T>0.$ Consequently, this identity in law still holds if $T$ is a positive r.v.~almost surely, and independent of $(B_t)$. In particular, replacing $T$ by $T_{m+1},$ and sampling at times $U_{(1)}, \ldots, U_{(m)},$ we obtain
\begin{align*}
  (B_{U_{(k)}})_{k=1}^m =_d (T'_{m+1})^{-\frac 1 2}(B_{U_{(k)}T'_{m+1}})_{k=1}^m =_d (T'_{m+1})^{-\frac 1 2}(B_{T'_{k}})_{k=1}^m,
\end{align*}
and therefore $\Prob{B_{U_1}>0, ..., B_{U_m}>0}=\Prob{B_{T'_1}>0, ..., B_{T'_m}>0}.$

Second step. In the second and final step we show how $(B_{T_k})_{k=1}^m$ may be expressed purely in terms of Poisson processes. As announced earlier, we now bring into play two more Poisson processes $\Pi, \Pi''$, and we'll redefine $\Pi'$ in terms of $\Pi$ though this does not affect the distribution of $\Pi'$. Let $\Pi=\{ (T_k, Y_k)\colon k\geq 1  \}$ denote a marked Poisson process on $[0, \infty)\times \{0, 1\}$ independent of $(B_t).$ We determine the distribution of $\Pi$ by requiring that its restriction $\{ T_k\colon k\geq 1 \}$ to the line is a Poisson process of intensity two independent of the i.i.d.~sequence of marks $(Y_k)_{k\geq 1},$ where $Y_1$ is Rademacher distributed. By the Marking Theorem for Poisson processes, cf.~\cite{Kingman1993},
\begin{align}
  \Pi'\coloneqq \{  T_k\colon Y_k=1 \}, \qquad  \Pi''\coloneqq \{  T_k\colon Y_k=-1\}
\end{align}
are two independent standard Poisson processes on the line whose points we denote by $(T'_k)$ and $(T''_k)$, respectively.

The process $(R_k)_{k\geq 1}$ defined by $R_k\coloneqq B_{T'_k}$ is a (continuous-space) random walk with i.i.d.~increments (as a short calculation with the joint characteristic function or, alternatively, Bochner's subordination theorem reveals). Each of the increments $(R_{k}-R_{k-1})_{k\geq 1}=(B_{T'_k}-B_{T'_{k-1}})_{k\geq 1}$ (where $R_0\coloneqq T'_0\coloneqq 0$) is distributed as
\begin{align}\label{eq:laplace_representation}
  B_{T'_1} =_d \sqrt{T'_1}B_1 =_d \frac{1}{\sqrt 2}(T'_1-T''_1),
\end{align}
where the first identity in distribution follows from the scaling property of Brownian motion, the second identity in distribution may easily be seen by computing the characteristic functions of both sides (cf.~\cite[equations (2.2.3) and (2.2.8)]{KotzKozubowskiPodgorski2001}). The law of $\sqrt 2 B_{T'_1}$ is also known as the standard Laplace distribution, cf.~\cite[Proposition 2.2.1]{KotzKozubowskiPodgorski2001}. In summary, we have shown that
$$
  \Prob{B_{U_1}>0, \ldots, B_{U_m}>0} = \Prob{T'_1>T''_1, \ldots, T'_m>T''_m }.
$$
\begin{remark}
  The event $\{ T'_1>T''_1, \ldots, T'_m>T''_m  \} = \{ R_1>0, \ldots, R_m>0  \}$ is nothing but the random walk $(R_k)$ surviving during its first $m$ steps. Since $(R_k)$ has i.i.d.~increments which are continuous and symmetric, this survival probability is given by $2^{-2m}\binom{2m}{m}$ according to Sparre Andersen's theorem on the fluctuations of random walks. However, we do not draw on this deep result of Andersen's, but instead continue with our elementary proof.
\end{remark}
Picture now the points in $\Pi$ as colored balls discovered in their order of appearance, with the $k$th ball black if $Y_k=1$ and white if $Y_k=-1$. Set $S_k \coloneqq Y_1+\cdots +Y_k,$ and let
\begin{align*}
  \tau_m\coloneqq \inf\left \{ k\geq m \colon \sum_{j=1}^k \mathbf 1_{\{Y_j=1  \}}  =m  \right \} 
\end{align*}
denote the first time when there are $m$ black balls drawn. Fix $w\geq 0.$ If $\tau_m=m+w$ then there are $w$ white and $m$ black balls among the first $m+w$ balls drawn with the last ball drawn black. Consequently,
\begin{align}
  \Prob{\tau_m=m+w} = \left (\frac 1 2\right )^{m+w}\binom{m+w-1}{m-1},\qquad w\geq 0,
\end{align}
and the distribution of $\tau_m-m$ is sometimes referred to as a negative binomial distribution. 	Given that $\tau_m=m+w$ the probability of  $S_1\geq 0, ..., S_{\tau_m-1}\geq 0$  is given by the ballot theorem with ties as $(m-w)/m,$ as expounded in Theorem~\ref{thm:ballot} and the subsequent comment. Thus
\begin{align*}
  & \Prob{ T'_1>T''_1, \ldots, T'_m>T''_m } \\
  &=  \Prob{ S_1\geq 0, ..., S_{\tau_m}\geq 0 }\\
  &= \sum_{w=0}^{m-1} \Prob{S_1\geq 0, ..., S_{m-1+w}\geq 0 \mid \tau_m=m+w}\Prob{\tau_m=m+w} \\
  &= 2^{-m}\sum_{w=0}^{m-1}  \frac{m-w}{m}2^{-w} \binom{m-1+w}{m-1},
  \intertext{and, since for $w\geq 1$ we have $\frac{m-w}{m}\binom{m+w-1}{m-1} = \frac{m+w-2w}{m}\frac{(m+w-1)!}{(m-1)!w!}=\binom{m+w}{m}-2\binom{m+w-1}{m},$}
  &= 2^{-m} \left (  1 + \sum_{w=1}^{m-1} \left [ 2^{-w}\binom{m+w}{m} -2^{-w+1}\binom{m+w-1}{m}\right ]     \right )\\
  &= 2^{-m} \left (  1 + \sum_{w=1}^{m-1}  2^{-w}\binom{m+w}{m} -\sum_{w=0}^{m-2} 2^{-w}\binom{m+w}{m}     \right )\\
  &= 2^{-2m+1}\binom{2m-1}{m-1}=2^{-2m}\binom{2m}{m}.
\end{align*}
We have just shown that the occupation time defined in~\eqref{def:occupation_time} has moment sequence $(2^{-2m}\binom{2m}{m})_{m\geq 1}.$ On the other hand, recall the beta integral $\int_0^1 x^{a-1}(1-x)^{b-1}dx=\Gamma(a)\Gamma(b)/\Gamma(a+b)$ for any $a, b>0,$ with the Gamma function defined by $\Gamma(x)\coloneqq\int_0^\infty e^{-t}t^{x-1}dt$ for $x>0,$ so the $m$th moment of the arcsine distribution with density~\eqref{eq:arcsine_density} is given by
\begin{align*}
  \pi^{-1}\int_0^1 x^{m-\frac 1 2}(1-x)^{-\frac 1 2}dx = \frac{\Gamma(m+\frac 1 2)\Gamma(\frac 1 2)}{\pi\Gamma(m+1)}=2^{-2m}\binom{2m}{m},
\end{align*}
where in the second equality we used $\Gamma(m+\frac 1 2)=\sqrt \pi(2m-1)!!/2^{m},$ $m\in\mathbb N,$ cf.~\cite[2.~in 8.339]{GradshteynRyzhik2007}, and $(2m-1)!!=1\cdot 3\cdots (2m-1)=(2m-1)!/(2\cdot 4\cdots (2m-2))=(2m-1)!/(2^{m-1}(m-1)!)$.
Since the distribution of a bounded real random variable is uniquely determined by its moment sequence, we conclude the proof.

\begin{remark}
  We have seen that the sampling method provides an elegant proof of the arcsine law. The sampling method is also flexible in that it works just as well in higher dimensions, and there leads to computations involving coupled Poisson processes. We believe that this approach may also be useful to study occupation times of Brownian motion in $\mathbb R^d,$ $d>1$, a topic with mostly open questions.
\end{remark}

\begin{remark}
  Fix a real number $\mu$ and let $B^{(\mu)}\coloneqq (B^{(\mu)}_t)_{t\geq 0}$ denote Brownian motion with drift $\mu$ defined by $B^{(\mu)}_t\coloneqq B_t+\mu t.$ Consider the time $\int_0^1 \mathbf 1_{(0, \infty)}( B^{(\mu)}_t)dt$ that this process spends above zero during $[0, 1].$ The law of the occupation time of $(B^{(\mu)}_t)$ has been studied in the context of option pricing in mathematical finance, e.g.,~\cite{Akahori1995, EmbrechtsRogersYor1995}. A formula for its distribution function may be found in~\cite[Theorem 1.1]{Akahori1995}, and a formula for its density is provided in~\cite[Equation (4a)]{EmbrechtsRogersYor1995}. Our derivation of the moments of the occupation time of $(B_t)$ may be readily adapted to this setting. The random walk obtained by sampling $B^{(\mu)}$ at the points $0<T'_1<T'_2<\cdots$ has i.i.d.~increments that are equal in distribution to $B^{(\mu)}_{T'_1}=_d (\kappa^{-1} T'_1-\kappa T''_1)/\sqrt 2$ with $\kappa\coloneqq (\sqrt{2+\mu^2}-\mu)/\sqrt 2$,  where the last equality in distribution is a well-known representation of the asymmetric Laplace distribution, cf.~\cite[equations (3.1.9) and (3.2.1)]{KotzKozubowskiPodgorski2001}. In complete analogy to the previous arguments, the time $\tau_m$ until we see the $m$th black ball has distribution $\Prob{\tau_m=m+w}=p^m(1-p)^w\binom{m+w-1}{m-1},$ $w\geq 0,$ where $p\coloneqq 1/(1+\kappa^2)$. For the $m$th moment of the sojourn time of $(B^{(\mu)}_t)$ we thus obtain
  \begin{align*}
    \Ex \left[ \left (\int_0^1 \mathbf 1_{\{ B^{(\mu)}_t>0\}} dt \right )^m \right ] &= p^m \sum_{w=0}^{m-1} (1-p)^w\frac{m-w}{m}\binom{m+w-1}{m-1}.
  \end{align*}
 In the case $\mu=0$ we recover the result of Theorem~\ref{thm:arcsine_law}.
\end{remark}

\textbf{Acknowledgements.} H.H.P.~thanks Leif D\"oring for stimulating discussions.
\newline

\textbf{Funding.} This work was supported by the DFG (German Research Foundation) under Grant 526069380.

\bibliographystyle{vancouver}
\bibliography{literature.bib}



\vfill\eject

\end{document}